\documentclass[12pt]{amsart}
\usepackage{amsmath}
\usepackage{amssymb}
\usepackage{fancybox}
\usepackage{eucal}
\usepackage{amsthm}
\usepackage{amscd}
\usepackage{hyperref}
\usepackage{wasysym}
\usepackage{url}

\usepackage{amssymb,}
\usepackage[]{amsmath, amsthm, amsfonts,graphicx, amscd,}
\usepackage[all,cmtip]{xy}
\input amssym.def \input amssym

\newtheorem {thm}{Theorem}[section]
\newtheorem{corr}[thm]{Corollary}

\newtheorem{prop}[thm]{Proposition}
\newtheorem{fact}[thm]{Fact}

\newtheorem {lem}[thm]{Lemma}

\theoremstyle{remark}
\newtheorem{rem}[thm]{Remark}
\newtheorem{np*}{Non-Proof}

\theoremstyle{definition}
\newtheorem{defn}[thm]{Definition}

\newtheorem{exam}[thm]{Example}

\def\Ind{\setbox0=\hbox{$x$}\kern\wd0\hbox to 0pt{\hss$\mid$\hss} \lower.9\ht0\hbox to 0pt{\hss$\smile$\hss}\kern\wd0}
\def\Notind{\setbox0=\hbox{$x$}\kern\wd0\hbox to 0pt{\mathchardef \nn=12854\hss$\nn$\kern1.4\wd0\hss}\hbox to 0pt{\hss$\mid$\hss}\lower.9\ht0 \hbox to 0pt{\hss$\smile$\hss}\kern\wd0}

\newcommand{\m}{\mathbb }
\newcommand{\mc}{\mathcal }
\newcommand{\mf}{\mathfrak }

\begin{document}

\title[Uniform bounding in partial differential fields]{Effective uniform bounding in partial \\ differential fields}

\author[James Freitag]{James Freitag*} \thanks{*This material is based upon work supported by the National Science Foundation Mathematical Sciences Postdoctoral Research Fellowship, award number 1204510. It is also based upon work supported by the National Science Foundation under Grant No. 0932078 000 while the authors
were in residence at the Mathematical Sciences Research Institute in Berkeley, California, during the Spring 2014 semester.} 
\author[Omar Le\'on S\'anchez]{Omar Le\'on S\'anchez}  
\address{freitag@math.berkeley.edu \\
Department of Mathematics\\
University of California, Berkeley\\
970 Evans Hall\\
Berkeley, CA 94720-3840 }
\address{oleonsan@math.mcmaster.ca \\
Department of Mathematics and Statistics \\
McMaster University \\
1280 Main St W \\
Hamilton, ON L8S 4L8}

\date{\today 
}

\begin{abstract}
Motivated by the effective bounds found in \cite{HPnfcp} for ordinary differential equations, we prove an effective version of uniform bounding for fields with several commuting derivations. More precisely,  we provide an upper bound for the size of finite solution sets of partial differential polynomial equations in terms of data explicitly given in the equations and independent of parameters. Our methods also produce an upper bound for the degree of the Zariski closure of solution sets, whether they are finite or not. 
\end{abstract}

\maketitle

\section{Introduction} 

Suppose we are given a system of partial differential polynomial equations over $\m Q$, 
\begin{align*} 
p_1 (x,y) =& \; 0 \\
p_2 (x,y) =& \; 0 \\
\vdots \, & \\
p_r (x,y) =& \; 0 \\
\end{align*} 
so that for some specific values of $y=(y_1,\ldots,y_{s})$ in some differentially closed field $(K,\Delta=\{\delta_1,\ldots,\delta_m\})$ of characteristic zero with commuting derivations, the number of solutions (in the variables $x=(x_1,\ldots,x_{n})$) in $K^n$ is finite. Can one bound the number of solutions in terms of the basic invariants of the differential polynomials $p_i$ without any reference to the selected values of $y$? More generally, without assuming finiteness of the solution set, can one bound the degree of its Zariski closure? In this paper we will answer these questions affirmatively, and give bounds which depend only  on the order, degree, and number of variables in the differential polynomials (and the number of derivations). 

Besides being a problem of foundational interest, this problem is intimately connected to the effective differential Nullstellensatz, and is also at the heart of applications of differential algebra to the so-called special points conjectures in number theory. We will give more details of these connections after discussing some of the history and difficulties of this problem. 

\begin{rem}
For the model theorist, the existence of this type of effective bounds implies, amongst other things, that the theory $\operatorname{DCF}_{0,m}$ has \emph{uniform bounding} which seems to be a new result in the case of fields with several commuting derivations (the effective bounds found in \cite{HPnfcp} imply uniform bounding  for the ordinary case, for a noneffective proof see \cite{MMP}). Consequently, since differentially closed fields are stable and eliminate imaginaries, a result of \cite{Sacks} implies that $\operatorname{DCF}_{0,m}$ has NFCP (the non-finite cover property).
\end{rem}

We now remark on the difficulties that arise (in the case of several commuting derivations) while trying to find effective bounds. The case of a single derivation was considered in \cite{HPnfcp}; let us briefly describe the methods of that paper. Assume $\Delta=\{\delta\}$, and let us consider the case of first-order differential equations. In this case, the problem can be restated as follows: Are there effective upper bounds for the size of finite sets of the form $Z=\{v\in V: (v,\delta(v))\in W\}$ where $V$ and $W$ are algebraic varieties?
\begin{fact}
(In the ordinary case) Let $V$ and $W$ be closed subvarieties of $K^n$ and $K^{2n}$, respectively. If $Z=\{v\in V:(v,\delta(v))\in W\}$ is finite, then $$|Z|\leq (\deg V)^{2^{\dim V}}(\deg W)^{2^{\dim V}-1}.$$
\end{fact}
The proof of this fact, as it appears in \cite{HPnfcp}, uses in an essential way the variety $B_1(V)$ defined as the Zariski closure of $\{(v,\delta(v))\in K^{2n}:v\in V\}$ equipped with its canonical projection $B_1(V)\to V$. The idea is that $W\cap B_1(V)$ can not project dominantly onto $V$; otherwise, $Z$ would be infinite. One then replaces $V$ with this projection, computes a bound for its degree, and repeats the process. This algorithm yields the desired bound. Note that this procedure translates the differential-algebraic problem into one purely of classical intersection theory, where Bezout's inequality can be used to compute degree bounds. (In Section \ref{prolongingthemagic} we do the degree computations of the relevant algebro-geometric objects in the case of several commuting derivations.)

It is worth mentioning that Hrushovski and Pillay were essentially only interested in the case when the Kolchin closed set $Z$ is finite; in other words, when the Zariski closure $\bar Z$ of $Z$ is zero dimensional. However, as they point out in Remark 3.2 of \cite{HPnfcp}, their algorithm also produces an upper bound for the degree of $\bar Z$, whether $Z$ is finite or not. Moreover, in the case that the dimension of the components of $\bar Z$ are bounded below by some positive integer, the algorithm yields a better bound for the degree of $\bar Z$. We point this out in Remark \ref{secondrem}~(2) (and extend it to the case of several commuting derivations in Corollary \ref{posdim}).

The key ingredient in the above algorithm is the fact that if $W\cap B_1(V)$ projects dominantly, then $Z$ is infinite (we state this formally in Fact \ref{ordi} below). This property, of ordinary differentially closed fields, \emph{does not} have a straightforward generalization to the case of fields with several commuting derivations, see Example \ref{example1} below. Essentially, the complications arise from the integrability conditions imposed by the commutativity of the derivations (we explain this in more detail in Section \ref{bur}). Moreover, a naive generalization of the above algorithm to the partial case (say $\Delta=\{\delta_1,\delta_2\}$) would say that if the set
$$Z_a= \{x\in K: \delta_1 x=x^2 \text{ and } \delta_2 x= x^3+a\}$$
is finite, then $|Z_a|\leq 3$; however, differentiating the above equations yields $x^4-2ax+\delta_1 a=0$ and so for an appropriate choice of $a$ we get exactly four solutions. Again, the issue here are the new algebraic relations that the commutativity of the derivations reveals after differentiating. Generally, there are algebraic relations that are not apparent until differentiating some additional number of times. One of the main ingredients in our bounds is to effectively determine how many times one has to differentiate to detect all such relations. We do this in Section \ref{bur} using results from \cite{Pierce2014fields}, and then we use this in Proposition \ref{goodfact} to prove our analogue of Fact \ref{ordi} in the case of several commuting derivations. We then combine the results of Sections \ref{prolongingthemagic} and \ref{bur} to prove our main theorem in Section \ref{count}. 

The rest of the introduction is devoted to explain the connections of our results to effective computational problems in differential algebra and effective results in number theory.

\subsection{The effective differential Nullstellensatz.}\label{diffynull} Given a system of (partial) differential equations $f_1 = 0 , \ldots  , f_r =0$ and a differential polynomial $f$, one can test if $f =0$ is a formal consequence of the given system. Specifically, there is an effective procedure which finds an expression for $f^k$, for some positive integer $k$, in terms of the elements $f_1, \ldots ,f_r$ and their derivatives, or shows that such an expression does not exist. The algorithm has two main steps: 

\begin{enumerate} 
\item Find an upper bound on the number of differentiations which might be required for such an expression. 
\item Find an upper bound on the degrees of the coefficients used in the expression and an upper bound for the number $k$. 
\end{enumerate}  

This procedure is called the \emph{effective differential Nullstellensatz}. The first attempt to solve this problem was considered in \cite{Seid} where it was (only) suggested how the bound for the first step of the algorithm could be obtained. The authors of \cite{golubitsky2009bound} gave a complete solution to the problem, but said of their solution: ``The differential elimination algorithms would be very useful for applications if there were faster versions of them". Special cases of the problem have been solved in the case of ordinary differential equations, where better bounds are available \cite{Argentine, QEG}. 

The bounds we establish in the course of this paper, specifically in Section \ref{bur}, can be used to improve the known bounds of the first step of the above algorithm. We will not elaborate on these ideas, but note that the reasoning along the lines of our srategy, and some of our results, are applied in \cite{GKO2014} to give new bounds for the effective differential Nullstellensatz. For instance, in the case of two derivations, the bound found in \cite{golubitsky2009bound} was $A(10,\max(n,h,d))$ where $A$ denotes the Ackermann function, $n$ is the number of variables, and $h$ and $d$ bound the order and degree of the differential polynomials, respectively. Using Lemma \ref{intwo} below, \cite{GKO2014} gives a bound which grows in $h$ as a tower of iterated exponentials of length $n$ and depends polynomially on $d$. This is an important practical development, because there are no values of $(n,h,d)$ such that $A(10,\max(n,h,d))$ can be calculated by current computers, while the iterated exponential bound may be practically calculated for many values of the inputs.

\subsection{Special points conjectures.} 
These type of effective bounds have also been applied to problems not a priori related to differential algebra. In particular, to the special points conjecture in number theory. For example, in \cite{HPnfcp}, Hrushovski and Pillay apply the effective bounds of the ordinary case to give effective bounds for the number of transcendental points contained in the intersection $X \cap \Gamma$ where $X$ is a subvariety of a semi-abelian variety $A$ which contains no translates of semi-abelian subvarieties and $\Gamma$ is a finite rational rank subgroup of $A$. 


Let us describe another recent application. We view $K^n$ as the moduli space of products of elliptic curves via their $j$-invariants. Numerical bounds on the size of finite sets given by various intersections of varieties with isogeny classes of transcendental points come via the effective bounds found in the ordinary case (together with a sharper bound for the degree of positive dimensional components as the one we point out in Remark \ref{secondrem}~(2) below). For instance, \cite{FSj} gives an effective upper bound for the degree of the Zariski closure of the intersection of an arbitrary Kolchin closed set in $K^n$ with the isogeny class of a tuple of transcendentals (various developments around the Andr\'e-Oort conjecture \cite{PilaAO} can be used to prove special cases of the finiteness result implicit in the next theorem; however, these methods are noneffective as they use the Pila-Wilkie counting theorem \cite{pila2006rational}): 


\begin{thm}\cite[Theorem 6.7]{FSj} \label{FSj}
Let $V \subseteq K^n$ be a Kolchin closed subset whose Zariski closure has dimension $d$ and $a$ be an $n$-tuple of transcendental points. Let 
$$\operatorname{Iso}(a) := \{ (b_1,\ldots,b_n) \in {K}^n ~\vert~ E_{a_i} \text{ and } E_{b_i} \text{are isogenous for } i \leq n\}$$
be the isogeny class of $a$.  If $W$ is the Zariski closure of $V \cap \operatorname{Iso}(a)$, then 
$$\deg W\leq (6\cdot \deg V)^{2^{3d}-1}.$$
In particular, when $V$ does not contain a weakly special subvariety (see \cite{PilaAO} for the definition) $W$ is zero dimensional, and one obtains a bound on the number of points in the intersection.
\end{thm} 

While the above number theoretic results only use the effective bounds obtained in the ordinary case, we expect that the partial differential bounds obtained in this paper will be used in future applications of a similar nature.

\

\noindent {\bf Acknowledgements.} The first author would like to thank Tom Scanlon for numerous useful conversations. Both authors would like to thank Rahim Moosa and Alexey Ovchinnikov for their suggestions on a previous draft of this paper. Both authors would like to thank Gal Binyamini for pointing out a gap in the preliminary version of the degree bounds. Finally, both authors would like to thank the anonymous referee whose comments led to significant improvements in the presentation of the paper.

\section{On the dimension and degree of prolongation spaces} \label{prolongingthemagic}

In our algorithm of Section \ref{count} that yields the desired uniform bounds, one needs to keep track of dimensions and degrees of certain algebraic varieties. Most notably, we will use the algebraic variety $B_\ell(V)$ whose definition we recall after fixing some notation. We should mention that our exposition here has been influenced by \cite{MSJETS}, where very general notions of \emph{prolongations} are developed (for instance with very general operators, including derivations as a special case). In the special case of differential operators, the prolongations of \cite{MSJETS} are related to those of \cite{Johnson}. In the appendix of \cite{Buium} the relationship between the prolongations in the sense of \cite{Johnson} and arc spaces from algebraic geometry \cite{mustata2002singularities} is explained. We will elaborate on this relationship later in this section. For the type of results we pursue in this paper, we find the notation and development of \cite{MSJETS} most convenient. 

Let $m$ be a positive integer and fix a differentially closed field $(K,\Delta=\{\delta_1,\ldots,\delta_m\})$ of characteristic zero with commuting derivations. Throughout the paper, $\m N = \{ 0, 1, 2, \ldots \} $ denotes the natural numbers. For each $\xi=(\xi_1,\dots,\xi_m)\in \m N^m$, we let $|\xi|=\xi_1+\cdots+\xi_m$ and $\delta^\xi=\delta_m^{\xi_m}\cdots\delta_1^{\xi_1}$. For $\ell\in \m N$, we let $\alpha_\ell=\binom{\ell+m}{m}$ and $\Gamma(\ell)=\{\xi\in \m N^m: |\xi|\leq \ell\}$. It is easy to check, by induction on $\ell$, that $|\Gamma(\ell)|=\alpha_\ell$. The $\ell$-th nabla map $\nabla_\ell:K^n \to K^{n\cdot\alpha_\ell}$ is defined by 
$$x\mapsto (\delta^{\xi}x)_{\xi\in \Gamma(\ell)},$$
where $x=(x_1,\dots,x_{n})$ are coordinates for $K^n$. For $\ell=1$, we use $\nabla$ instead of $\nabla_1$.

We are interested in counting the number of points of finite subsets of $K^n$ of the form $Z=\{v\in V: \nabla_\ell (v)\in W\}$ where $V$ and $W$ are closed subvarieties of $K^n$ and $K^{n\cdot\alpha_\ell}$, respectively. For our methods of proof, which follow those in \cite{HPnfcp}, it will be useful to consider the following algebraic variety:

\begin{defn} 
Let $V\subseteq K^n$ be a closed subvariety. For each $\ell\in \mathbb N$, we let $B_\ell (V)\subseteq K^{n\cdot \alpha_{\ell}}$ be the Zariski closure of $\{\nabla_\ell(v)\in K^{n \cdot \alpha_\ell}: v\in V\}$. We let $\pi_\ell:B_\ell(V)\to V$ be projection onto the first $n$ coordinates.
\end{defn}

\begin{rem}
One can also define $B_\ell(V)$ as follows. Let $\mathcal I_\ell(V/K):=\{f\in K\{x\}: f(V)=0 \text{ and } \operatorname{ord}(f)\leq \ell\}$ where $K\{x\}$ denotes the ring of differential polynomials over $K$ in the variables $x=(x_1,\ldots,x_{n})$. Given $f\in K\{x\}$ with $\operatorname{ord}(f)\leq \ell$, we let $f^\xi$ be the polynomial over $K$ obtained from $f$ by replacing $\delta^\xi x$ by $x^{\xi}$, where $(x^\xi)_{\xi\in \Gamma(\ell)}$ are coordinates of $K^{n\cdot\alpha_\ell}$. Then the defining ideal of $B_\ell(V)$ is given by 
$$\{f^\xi \in K[x^\xi:\xi\in \Gamma(\ell)]: f\in \mathcal I_\ell(V/K)\}.$$
\end{rem}

We now show that the dimension of $B_\ell(V)$ can be expressed in terms of $
V$ and $\alpha_\ell$.

\begin{lem} \label{goodcomp} 
If $V\subseteq K^n$ is an irreducible subvariety, then $B_\ell(V)$ is irreducible and $\dim (B_\ell (V))= \alpha_\ell\cdot \dim (V)$. 
\end{lem} 
\begin{proof} 
This follows from Kolchin's irreducibility theorem \cite[Chap. IV, \S17]{KolchinDAAG}. Since $V$ is irreducible in the Zariski topology, $V$ is also irreducible in the Kolchin topology, and so $B_\ell(V)$ is irreducible being the Zariski closure of a Kolchin-irreducible variety (the graph of a differential-algebraic function with Kolchin-irreducible domain). The computation of $\dim (B_\ell(V))$ follows from the fact that the differential transcendence degree of the differential function field of $V$ is equal to $\dim V$. 
\end{proof} 

Let us observe the following fact which allows us to restrict certain arguments about $B_\ell(V)$ to the case when $V$ is irreducible.

\begin{lem}\label{irrb}
Let $V$ be a closed subvariety of $K^n$. If $(V_i)_{1\leq i\leq p}$ are the irreducible components of $V$, then the $B_\ell(V_i)$'s are the irreducible components of $B_\ell(V)$.
\end{lem}
\begin{proof}
We have 
$$B_\ell(V)=B_\ell(V_1)\cup\cdots\cup B_\ell(V_p).$$
By Lemma \ref{goodcomp}, the $B_\ell(V_i)$'s are irreducible. It suffices to show that the decomposition of $B_\ell(V)$ is irredundant. If it were not irredundant, for some $i$ we would have $B_i(V)\subseteq \cup_{j\neq i}B_\ell(V_j)$ and, taking projections, this would imply that $V_i\subseteq \cup_{j\neq i}V_j$ which is impossible.
\end{proof}

We now aim to show that 
\begin{equation}\label{degbound}
\deg B_\ell(V)\leq (\deg V)^{\alpha_\ell}.
\end{equation}
But before doing so, let us remind the reader of the notion of degree (and its basic properties) for affine algebraic varieties. If $V\subseteq K^n$ is an irreducible subvariety of dimension $d$, then the \emph{degree} of $V$ is defined as
\begin{align*}
\deg V =&  \max \{|V\cap H_1\cdots\cap H_d|: \text{ each } H_i \text{ is a hyperplane of } K^n \\
 & \qquad \text { and the intersection is finite}\}
\end{align*}
When $V$ is not irreducible, $\deg V$ is defined as the sum of the degrees of its irreducible components. 

It follows immediately from the definition of degree that $\deg (K^n)=1$, for any $n$, and that the degree of a finite set equals its cardinality.  
In the special case of an irreducible hypersurface $H \subseteq K^n$, one can show that the degree of $H$ is equal to the degree of a generator of the ideal of polynomials over $K$ vanishing on $H$. 

We will make use of the following results of Heintz \cite{Heintz}:

\begin{fact}\label{Bezout} 
\begin{enumerate}
\item If $\phi:K^n\to K^{n'}$ is an affine linear map (e.g., a projection) and $V\subseteq K^n$ is a subvariety, then $\deg \,\phi(V)\leq \deg V$. 
\item (Bezout's Inequality) If $V$ and $V'$ are closed subvarieties of $K^n$, then $$\deg (V\cap V')\leq \deg V \cdot \deg V'$$
\end{enumerate}
\end{fact}

The proof of inequality (\ref{degbound}) requires more work than the computation of $\dim (B_\ell (V))$, mainly because determining its defining equations is a nontrivial problem. The proof we present here follows the strategy of Fact 3.6 of \cite{HPnfcp}; that is, we go via the theory of prolongations spaces.

Let us recall the notion of prolongation for algebraic varieties (our presentation is informed by Sections 3 and 4 of \cite{MSJETS}). For $\ell\in \mathbb N$, we let 
$$K_\ell=K[\epsilon]/(\epsilon)^{\ell+1}$$
where $\epsilon=(\epsilon_1,\dots,\epsilon_m)$ is a tuple of variables. Note that $K_\ell$ has the \emph{standard} $K$-algebra structure $s:K\to K_\ell$, but also the \emph{exponential} $K$-algebra structure $e:K\to K_\ell$ given by
$$a\mapsto \sum_{\xi\in \Gamma(\ell)}\frac{1}{\xi_1!\cdots \xi_m!}\delta^{\xi}(a)\epsilon^\xi,$$
where $\epsilon^\xi=\epsilon_1^{\xi_1}\cdots\epsilon_m^{\xi_m}$. To distinguish between these two structures we write $K_\ell^e$ to denote the exponential structure.

\begin{defn}\label{defpro}
Given an algebraic variety $V$ over $K$, the \emph{$\ell$-th prolongation} $\tau_\ell V$ of $V$ is the algebraic variety given by the Weil restriction of $V \times_K K_\ell^e$ from $K_\ell$ to $K$. Note that the base change $V\times_K K_\ell^e$ of $V$ is with respect to the exponential $K$-algebra structure, while the Weil restriction is with respect to the standard $K$-algebra structure. 
\end{defn}

The prolongation $\tau_\ell V$ has the characteristic property that for any $K$-algebra $R$, if we set $R_\ell:=K_\ell\otimes R$, then the $R$-points of $\tau_\ell V$ can be identified with the $R_\ell$-points of $V\times_K K_\ell^e$. Via this identification, for any morphism $f: V\to W$  we have a natural induced morphism $\tau_\ell f:\tau_\ell V\to \tau_\ell W$, and the exponential structure $e:K\to K_\ell$ induces a map $\nabla_\ell:V(K)\to \tau_\ell(V)(K)$. Also, using the residue map $K_\ell\to K$ and identifying $\tau_0V$ with $V$, we obtain the projection map $\pi_\ell:\tau_\ell V\to V$ of which $\nabla_\ell$ is a section (on $K$-points).

Prolongations of affine algebraic varieties are again affine. In terms of equations, if $V\subseteq K^n$ is a closed subvariety with defining ideal $\mathcal I\subseteq K[x]$ and $x=(x_1,\ldots,x_{n})$, then $\tau_\ell V$ is the subvariety of $K^{n\cdot\alpha_\ell}$ defined as follows: Let $\bar x=(x^\xi)_{\xi\in \Gamma(\ell)}$ be coordinates for $K^{n\cdot \alpha_\ell}$, where we identify $x=x^{\bf 0}$. For each $f\in \mathcal I$, let $f^e\in K_\ell[x]$ be the polynomial obtained by applying $e$ to the coefficients of $f$, and compute 
\begin{equation}\label{eq11}
f^e(\sum_{\xi\in\Gamma(\ell)}x^\xi \epsilon^\xi)=\sum_{\xi\in\Gamma(\ell)}f^{\xi}(\bar x)\epsilon^\xi
\end{equation}
in the polynomial ring $K_\ell[\bar x]=\bigoplus_{\xi\in\Gamma(\ell)} K[\bar x]\epsilon^\xi$. Then $\tau_\ell V$ is the zero set of $f^\xi=0$ as $\xi$ ranges in $\Gamma(\ell)$ and $f$ ranges in $\mathcal I$ (in fact it suffices to range $f$ in a set of generators of the ideal $\mathcal I$). The computation in (\ref{eq11}) yields that every $f^\xi$ is obtained by replacing $\delta^\xi x$ by $x^\xi$ in the differential polynomial $\delta^\xi f$. Note that with respect to these coordinates, the nabla map $\nabla_\ell: V\to \tau_\ell V$ is given by $\nabla_\ell(x)=(\delta^\xi x)_{\xi\in\Gamma(\ell)}$, and the projection map $\pi_\ell:\tau_\ell V\to V$ is given by $\bar x\mapsto x$.

\begin{rem}
The above construction of the prolongation is a particular case of the general theory of prolongation spaces developed by Moosa and Scanlon in \cite{MSJETS}. In their general setting they fix a ring $R$ and an arbitrary finite free $R$-algebra with basis (see Remark 3.2 of \cite{MSJETS}). In our case, $R=\mathbb Q$ and the finite $\mathbb Q$-algebra is $\mathbb Q[\epsilon_1, \ldots , \epsilon _m ]/(\epsilon_1, \ldots , \epsilon_m)^{n+1}$. Thus, we can (and will) freely apply the results of \cite{MSJETS}.
\end{rem}

Let $V\subseteq K^n$ be an irreducible subvariety and $\ell \in \m N$. Since $\nabla_\ell(v)\in\tau_\ell V$ for all $v\in V$, we have that $B_\ell(V)$ is an irreducible subvariety of $\tau_\ell V$. In general, $\tau_\ell V$ might not be irreducible and its dimension might be larger than $\dim (B_\ell (V))$ as the following example shows:

\begin{exam}
Consider the ordinary case $\Delta=\{\delta\}$. Let $V \subseteq K^3$ be given by the equation $x ^4 +y^4+z^4 =0$. By Proposition \ref{goodcomp}, $\dim (B_3 (V)) = 8$. In this case, since the variety is defined over the constants, $\tau_3V$ is the same as what in \cite{mustata2002singularities, mustatanotes} Musta{\c{t}}{\u{a}} calls the $3^{\text{rd}}$-jet space of $V$. It follows from the generalities in Examples 1.9 through 1.11 of \cite{mustatanotes} that in this case $\dim( \tau _3 V )=9$. Here we simply point out why the dimension is at least 9. Since $\tau_3V$ can be identified with the $K[\epsilon]/(\epsilon)^4$-points of $V$, any point of the form $(a_0+a_1 \epsilon + a_2 \epsilon^2 +a_3 \epsilon^3 , b_0+b_1 \epsilon + b_2 \epsilon^2 +b_3 \epsilon^3 , c_0+c_1 \epsilon + c_2 \epsilon^2 +c_3 \epsilon^3 )$ is in $\tau_3V$. Correspondingly, the fiber of $\pi_3: \tau_3V \rightarrow V$ over $(0,0,0)$ has dimension $9$. 
\end{exam}

\begin{rem}
Various examples of the above kind appear in \cite{mustata2002singularities}. Furthermore, in the ordinary case $\Delta=\{\delta\}$, the dimension of $\tau_\ell(V)$ grows in a way which is controlled by an invariant called the \emph{log canonical threshold} $lct(V,K^n)$ (see \cite{mustata2002singularities} for the definition). When $V$ is defined over the constants, and so $\tau_\ell(V)$ coincides the $\ell^{\text{th}}$-jet space of $V$, we have the following formula (which is a special case of \cite[Corollary 0.2]{mustata2002singularities}):
$$lct (V, K^n ) = n - \max _{\ell \in \m N} \frac{\dim \tau_\ell (V)}{\ell+1}. $$
In \cite{rosen2008prolongations}, Rosen proved that, for arbitrary $V$, $\tau_\ell(V)$ and the $\ell^{\text{th}}$-jet space of $V$ are isomorphic as $\delta$-varieties. Hence, the above formula holds for general $V$ (i.e., not necessarily defined over the constants).

\end{rem}
Conveniently, when $V$ is smooth we do have that $\tau_\ell V=B_\ell(V)$, see \cite[\S4.3]{MSJETS}. This behavior of $\tau_\ell V$ on the non-singular locus of $V$ will allow us to show that in general $B_\ell (V)$ is an irreducible component of $\tau_\ell V$, and so we will obtain some information about the degree of $B_\ell(V)$ from that of $\tau_\ell(V)$.

Let us give an upper bound for the degree of $\tau_\ell(V)$ in the case when $V$ is a hypersurface.

\begin{lem}\label{prope}
If $H\subseteq K^n$ is an affine hypersurface, then $\deg \tau_\ell(H) \leq (\deg H)^{\alpha_\ell}$.
\end{lem}
\begin{proof} 
Let $f$ be a polynomial of degree $d=\deg H$ that generates the ideal of $H$ over $K$. By (\ref{eq11}) and comments after, the prolongation $\tau_\ell(H)$ is defined by $f^\xi(\bar x)=0$ as $\xi$ ranges in $\Gamma(\ell)$, where $f^\xi(\bar x)$ is obtained by replacing $\delta^\xi x$ by $x^\xi$ in the differential polynomial $\delta^\xi f$. Thus, as polynomials in the variables $\bar x$, each of the $f^\xi$'s has degree at most $d$. So, by Bezout's inequality (see Fact \ref{Bezout}), the degree of their intersection, which equals $\tau_\ell (H)$, has degree at most $(\deg H) ^ {\alpha_\ell}$.
\end{proof}

We now show that $B_\ell(V)$ is an irreducible component of $\tau_\ell(V)$ \footnote{At this stage of the argument in Fact 3.6 of \cite{HPnfcp}, the authors argue that, when $V$ is a hypersurface, $B_\ell(V)$ is a component of $\tau_\ell(V)$ because they have the same dimension. However, they did not take into account the possibility that $V$ may be singular and so that $B_\ell(V)$ and $\tau_\ell(V)$ may have different dimensions. We have corrected the argument here using a reduction to the non-singular locus.}. 

\begin{prop}
Let $\ell\in \m N$. If $V\subseteq K^n$ is an irreducible subvariety, then $\tau_\ell (V)$ has only one irreducible component projecting dominantly onto $V$ and the dimension of this component is $\alpha_\ell \cdot \dim V$. 
\end{prop}
\begin{proof}
By \cite[Corollary 4.18]{MSJETS}, if $X$ is a smooth irreducible variety over $K$, then $\tau_\ell(X)$ is smooth and irreducible. By \cite[Proposition 4.6]{MSJETS}, if $f: X \rightarrow V$ is an \'etale morphism, then $\tau_\ell f : \tau_\ell X \rightarrow \tau_\ell V$ is \'etale. Take $X$ to be the smooth locus of $V$. Since open immersions are \'etale, it follows that $\tau_\ell(X)\to X$ is the restriction of $\tau_\ell(V)\to V$ to $X$, and thus there is exactly one component of $\tau_\ell V$ which projects dominantly onto $V$. 

For the dimension, we need only calculate the dimension of the fiber of $\pi_\ell: \tau_\ell V \rightarrow V$ over a generic point (or any smooth point) of $V$, because then, by the fiber-dimension theorem, the dimension of the component in question must be the dimension of the fiber plus the dimension of $V$. As above, take $X$ to be the smooth locus of $V$. By \cite[Proposition 4.6]{MSJETS}, it suffices to do the calculation of the dimension of the fiber in $\tau_\ell X$. This reduces the computation to the case when the variety is smooth. Now the dimension calculation essentially follows from \cite[Proposition 4.17 (b)]{MSJETS} which we now explain with some additional detail. For the remainder of the proof we follow the notation and conventions of \cite{MSJETS}.

Let $\mc E$ be the ring scheme which associates to any ring $R$ the ring $$\mc E (R) = R [\epsilon] / (\epsilon)^{\ell+1},$$ 
where $\epsilon=(\epsilon_1,\dots,\epsilon_m)$. Let $e: K \rightarrow \mc E (K) $ be the homomorphism 
$$a\mapsto \sum_{\xi\in \Gamma(\ell)}\frac{1}{\xi_1!\cdots \xi_m!}\delta^{\xi}(a)\epsilon^\xi,$$
where $\epsilon^{\xi}=\epsilon_1^{\xi_1}\cdots\epsilon_m^{\xi_m}$. The maximal ideal $\mf m$ of $\mc E (K)$ is $(\epsilon_1, \ldots , \epsilon_m)$ and the highest power of $\mf m$ which is nonzero is $\mf m ^ {\ell }$. Let $\mc E_ i : = \mc E / \mf m ^{i+1}$. Now, we have the sequence 
$$\xymatrix{ \mc E = \mc E  _ {\ell } \ar[r] ^{\rho_{\ell-1}} & \mc E  _{\ell -2} \ar[r] ^{\rho_{\ell-2}}  & \cdots \ar[r] ^{\rho_{1}}  & \mc E  _1 \ar[r]^{\rho_{0}}  & \mc E _0}$$
where $\rho_i$ denotes the quotient map. Let $e_i : = \rho_i \circ \ldots \circ \rho_{\ell -1} \circ e$. Now, assuming $X$ is a smooth irreducible variety over $K$, we have an induced sequence of surjective morphisms: 
$$\xymatrix{ \tau (X, \mc E , e )  \ar[r]  & \tau (X, \mc E_{\ell -1}  , e_{\ell -2} )  \ar[r]  & \cdots \ar[r]   & \tau (X, \mc E_{1 }  , e_{1} )  \ar[r] & X }.$$ Note that in this particular setting, the map $e_0$ is the identity, so we have replaced $X^{e_0}$ appearing in the diagram of \cite[Proposition 4.17]{MSJETS} with $X$. 

For each $i =0 , \ldots , \ell-1$, $\tau ( X, \mc E _{i+1} , e_{i+1})  \rightarrow  \tau ( X, \mc E _{i} , e_{i})$ is a torsor of the tangent bundle of $X$ to the $m_i^{\text{th}}$ power, where $m_i =\dim_K ( \mf m^{i+1} / \mf m^{i+2})$ \footnote{There is a typo in the definition of $m_i$ in \cite{MSJETS}.} and the dimension is in the sense of $K$-algebras. By the fiber-dimension theorem, $$\dim (\tau ( X, \mc E _{i+1} , e_{i+1})) = \dim (\tau ( X, \mc E _{i} , e_{i})) + m_i \cdot \dim (X).$$ 
Thus, 
$$\dim ( \tau (X, \mc E , e )) =\left(1+ \sum _ {i =0}^{\ell-1} m_i \right)\cdot \dim(X).$$ 
Since dimension of $K$-algebras is additive in sequences of surjective homomorphisms,  we have that $1+\sum _ {i =0}^{\ell-1} m_i = \dim _K \mc  E(K) = \alpha _ \ell$, completing the proof. 
\end{proof}

\begin{corr}\label{patchy} 
If $V\subseteq K^n$ is irreducible, then $B_\ell(V)$ is the unique component of $\tau_\ell(V)$ that projects dominantly onto $V$.
\end{corr} 
\begin{proof} 
The irreducible subvariety $B_\ell(V)$ of $\tau_\ell(V)$ projects dominantly onto $V$ and, by Lemma \ref{goodcomp}, has the same dimension as the unique component of $\tau_\ell V$ with this property. 
\end{proof} 

Proposition \ref{prope}, together with Lemma \ref{irrb} and Corollary \ref{patchy}, show that if $H$ is a hypersurface  of $K^n$, then $\deg B_\ell(H)\leq (\deg H)^{\alpha_\ell}$. The next lemma shows that this degree bound holds for arbitrary affine varieties \footnote{The first author thanks Tom Scanlon for discussions which led to some of the ideas here. Both authors thank Gal Binyamini and the anonymous referee for pointing out a gap in the original argument of a preliminary version of this paper.}. 

For convenience of notation, we assume that our varieties are defined over a differentially closed subfield $F$ of $K$ such that $(K,\Delta)$ is universal over $(F,\Delta)$; that is, $K$ contains a differential generic point for each Kolchin closed set defined over $F$. Alternatively, one could also work in a differentially closed field universal over $K$. 

Let $(a_{i,j}) _{1\leq i\leq d+1,\, 1\leq j\leq n}$ be a tuple of independent transcendentals over $F$. We will write $\bar a$ to refer to the entire tuple of elements. Consider the map $\phi_{\bar a }: K^n \rightarrow K^{d+1}$ given by $$y = (y_1, \ldots , y_n ) \mapsto \left(\sum _{j=1}^n a_{1,j} y_j , \ldots , \sum _{j=1}^n  a_{d+1,j}  y_j \right).$$ We will call such a map an \emph{$F$-generic linear map}.

\begin{lem} \label{degreehyp}  Let $V \subseteq K^n$ be an irreducible variety of dimension $d<n-1$. Let $\bar a$ be a tuple of independent $\Delta$-transcendentals over $F$, and $\phi_{\bar a}$ be as above. Then, for all $\ell \in \m N$, $\tau _\ell \phi _{\bar a}$ is a degree preserving birational map between $B_\ell (V) $ and $ B_\ell(H)$ where $H $ is the Zariski closure of $ \phi_{\bar a} (V) \subseteq K^{d+1}$. 
\end{lem} 

\begin{proof} 
Let $\Theta(\ell)=\{\delta_1^{k_1}\cdots\delta_m^{k_m}:\, k_1+\cdot+k_m\leq \ell\}$. We enumerate the coordinate functions of $B_ \ell (H)$ by $\{x_ {i, \theta } : \, 1\leq i \leq d+1, \, \theta \in \Theta (\ell ) \}$, and the coordinate functions of $B_ \ell (V)$ by $\{y_ {i, \theta } : \, 1\leq i\leq n, \, \theta \in \Theta (\ell ) \}$.  When $\theta$ is the trivial differential operator we simply write $x_i$ for the coordinate on $H$ or $B_ \ell (H)$.

Among the coordinate functions of $H$, $x_1, \ldots , x_d$ are algebraically independent and $x_{d+1}$ is algebraically dependent over $x_1, \ldots , x_d$. Thus, $\{x_ {i, \theta } : \, 1\leq i\leq d, \, \theta \in \Theta (\ell ) \}$ is a transcendence basis for the function field $K (B_ \ell (H) )$. Let $\tau _ \ell \phi_{\bar a} ^* : K (B_ \ell (H) ) \rightarrow K (B_ \ell (V) )$ be the embedding of fields induced by $\tau _ \ell \phi _ { \bar a }$. Then $\{ \tau _ \ell \phi_{\bar a}^* (x_ {i, \theta } ) : \, 1\leq i\leq d, \, \theta \in \Theta (\ell ) \}$ is a transcendence basis for $K (B_  \ell (V))$ (this follows from the genericity of $\bar a$ and Lemma \ref{goodcomp}). 

Let $\preceq$ be the partial order on $\Theta (\ell)$ given by $\delta_1 ^{n_1} \cdots \delta_m ^{n_m} \preceq \delta_1 ^{k_1} \cdots \delta_m ^{k_m}$ if and only if $n_i \leq k_i$ for each $i=1, \ldots , m$. Now fix $\theta _1 \in \Theta ( \ell ) \setminus \Theta ( \ell -1) $, and define $K_ {\theta _1}$ to be the field generated by $\{y_ {i, \theta }  : \; 1\leq i\leq n, \, \theta \in \Theta (\ell ) \text{ and } \theta \preceq \theta _1\}$ over $K (\tau _ \ell \phi_{\bar a}^* (x_ {i, \theta } ) : \, 1\leq i\leq d, \, \theta \in \Theta (\ell ) ) . $ By construction, the field generated by the union of the fields $K_{\theta _1}$ as $\theta _1 $ ranges over $\Theta ( \ell ) \setminus \Theta (\ell -1)$ is all of $K ( B_ \ell (V))$. Now, for each $\theta_1 \in  \Theta ( \ell ) \setminus \Theta ( \ell -1)$, because $a_{d+1, 1}, \ldots , a_{d+1, n}$ are independent $\Delta$-transcendentals over $F \langle a_ {i,j} :\, 1\leq i\leq d, \, 1\leq j\leq n\rangle$, the function $\tau _\ell \phi _{ \bar a } ^* (x_{d+1, \theta_1 })$ is a generator of $K_{\theta _1 }$ over $K (\tau _ \ell \phi_{\bar a}^* (x_ {i, \theta } ): \, 1\leq i\leq d, \, \theta \in \Theta (\ell ))$. This follows by the primitive element theorem \cite[Chapter II, Theorem 19]{ZariskiSamuel} because $x_{d+1, \theta_1 }$ is a linear combination of the generators of $K_{\theta_1}$ over $K (\tau _ \ell \phi_{\bar a}^* (x_ {i, \theta } ) : \, 1\leq i\leq d, \, \theta \in \Theta (\ell ))$ which is generic over $F \langle a_ {i,j}:\, 1\leq i\leq d, \, 1\leq j\leq n\rangle $.

As the union of the subfields $K_ {\theta _1}$ generates $K ( B_ \ell (V))$ and each of these is contained in the image of $K ( B_ \ell ( H))$ under $\tau _ \ell \phi _{\bar a } ^* $, we must have $K ( B_ \ell (V)) \cong K ( B_ \ell ( H))$, establishing the birationality of the varieties. 

We now show that $\deg (B_{\ell} (V)) = \deg (B_ \ell (H))$. Let $\phi _ {\bar b}: B_ \ell (H) \rightarrow K^{d \cdot \alpha _ \ell }$ be the $F \langle \bar a \rangle $-generic linear map induced by a tuple $\bar b$ of $F \langle \bar a \rangle $-transcendentals. Then we claim that $\phi_{\bar b} : B_ \ell (H)  \rightarrow K^{d \cdot \alpha _ \ell }$ and $\phi_{\bar b } \circ \tau_\ell \phi _{\bar a}: B_\ell(V) \rightarrow K^{d \cdot \alpha _ \ell }$ are finite surjective morphisms. To see this, note that any sufficiently general linear projection to $ K^{d \cdot \alpha _ \ell }$ from a variety of dimension $d \cdot \alpha _ \ell $ satisfies the conditions for Noether normalization; indeed, as noted in the proof of Theorem 13.3 of \cite{eisenbud1995commutative} and the ensuing discussion on page 284, the generators in the proof of Noether normalization can be chosen to be linear combinations of the generators of the ring and in fact any suitably general linear combination will work. The linear combinations in our maps are generic by the independence of $\bar b$ over $K \langle \bar a \rangle$. 

Now, by \cite[Proposition 8.3~(1)]{Harm} \footnote{In the notation of \cite[Proposition 8.3]{Harm}, the degree of the morphism $\phi |_X $ refers to the number of points in a generic fiber of the map; of course, this number is equal to the degree of the induced extension of function fields.}, we have $$\deg (B_{\ell} (V)) = [K(B_ \ell (V)) : K (\m A ^{d \cdot \alpha _ \ell })] $$ and $$\deg (B_ \ell (H)) = [K(B_ \ell (H)) : K (\m A ^{d \cdot \alpha _ \ell} )]. $$ But $K(B_ \ell (V)) \cong K(B_ \ell (H))$, and the degree of a field extension is an isomorphism invariant, so $\deg (B_{\ell} (V)) = \deg (B_ \ell (H))$. 
\end{proof}

The above lemma yields the desired bound:
$$\deg B_\ell (V) \leq  (\deg V) ^ {\alpha_\ell}.$$

Let us remark that we have \emph{not} yet established any degree bound for $\tau _ \ell (V)$ for a general affine variety $V$ (we have only covered the hypersurface case in Lemma \ref{prope}). The argument in Proposition \ref{degreehyp} for the degree bound of $B_\ell(V)$ does not adapt to $\tau _ \ell(V)$ in general. The problem with the argument is that for a hypersurface $H$ of $K^{\dim V +1}$, the dimension of $\tau_ \ell(H)$ can be at most $\alpha_\ell(\dim V+1)$. Such restriction on the growth of the dimension precludes the existence of a map $\phi _ {\bar a}$ as in Proposition \ref{degreehyp} for which the induced map on prolongations is birational on each component, since it would imply that $\dim \tau_\ell(V)$ is bounded by $\alpha_\ell(\dim V+1)$ which is not true for general varieties. For example, in the ordinary case $(K,\delta)$, consider the curve $C \subseteq K^3$ whose ideal has generators $2x^7 + y^7+z^7$ and $x^4 y^4 + y^4 z^4 + x^4 z^4 $. In this case, $\tau _ 3 (C)$ has dimension at least $9$, because over the point $(0,0,0) \in C$ the fiber of $\tau _3 (C)$ is a copy of $K^9$ (in fact, $\tau _3 (C)$ has one component of dimension $8$ and one of dimension $9$); but $\alpha_3(\dim C+1)=8$. So, the line of reasoning from Lemma \ref{degreehyp} can not be seamlessly extended to yield that $(\deg V)^{\alpha_\ell}$ is an upper bound for the degree of $\tau_\ell (V)$. There is, however, an upper bound for the degree of $\tau_\ell(V)$ that depends only on $n$, $\alpha_\ell$ and $\deg V$, it is just not as practically useful as the one for the degree of $B_\ell(V)$:

\begin{prop}
Let $V\subseteq K^n$ be a closed subvariety and $\ell\in\m N$. There is a positive integer $D=D(n,\deg V)$ such that $\deg \tau_\ell(V)\leq D^{n\cdot \alpha_\ell+1}$.
\end{prop}
\begin{proof}
Let $D$ be a positive integer such that if $f_1,\dots,f_s\in K[x_1,\dots,x_{n}]$ are of degree at most $\deg V$ then $\sqrt{f_1,\dots,f_s}$ is generated by polynomials of degree at most $D$ (it is well known that such a $D$ exists and that it only depends on $n$ and $\deg V$, see for instance \cite{vandenDriesbounds} or \cite{Laplagne}). By Proposition 3 of \cite{Heintz}, there are polynomials $f_1,\dots,f_s$ of degree at most $\deg V$ such that the ideal of $V$ over $K$ is given by $\mathcal I(V/K)=\sqrt{f_1,\dots,f_s}$. By the choice of $D$, there are polynomials $g_1,\dots,g_r$ of degree at most $D$ such that $\mathcal I(V/K)=(g_1,\dots,g_r)$. The prolongation $\tau_\ell(V)$ is then given by the zero set  of $g_i^\xi=0$ as $\xi$ ranges in $\Gamma(\ell)$ and $i=1,\dots,r$ (see (\ref{eq11}) and the discussion after). Note that each $g_i^\xi$ has degree at most $D$. Now, by Kronecker's theorem (see \cite[Chap. VII, \S17]{Ritt}), there are polynomials $(h_i)_{i=1}^{n\cdot\alpha_\ell+1}$ of degree at most $D$ such that $\sqrt{(g_i^\xi)_{1\leq i\leq r, \xi\in \Gamma(\ell)}}=\sqrt{(h_i)_{i=1}^{n\cdot\alpha_\ell +1}}$. Finally, by Bezout's inequality, we have that
$$\deg \tau_\ell(V)\leq \prod_{i=1}^{n\cdot\alpha_\ell +1} \deg h_i\leq D^{n\cdot\alpha_\ell +1}.$$
\end{proof}

\section{The burden of commutativity}\label{bur}

In this section we discuss the proper setup to prove (in Proposition \ref{goodfact}) the analogue of the following property of ordinary differentially closed fields in our setting with finitely many commuting derivations. As we mentioned in the introduction, this property is at the heart of the proof of effective uniform bounding for ordinary differential fields (cf. Fact 3.7 of \cite{HPnfcp}).

\begin{fact}\label{ordi}
Suppose $(K,\delta)$ is differentially closed. If $V\subseteq K^n$ is an irreducible subvariety and $W$ is a subvariety of $K^{2n}$ such that $W\cap B_1(V)$ projects dominantly onto $V$, then for any nonempty Zariski open subset $U$ of $V$ there is $v\in U$ such that $(v,\delta v)\in W$.
\end{fact} 

\begin{rem}
This fact seems to have been the original motivation for the development of the so-called geometric axioms for ordinary differentially closed fields \cite{PiercePillay}.
\end{rem}

The situation is very different in the setting of partial differential equations. For instance, differential-algebraic varieties of the form $\{v\in V:\nabla(v)\in W\}$, where $V$ and $W$ are algebraic varieties, might be finite (even empty) even when $W\cap B_1(V)$ projects dominantly onto $V$. These situations can be witnessed in basic examples like the following:

\begin{exam}\label{example1}
Suppose $(K,\delta_1,\delta_2)$ is differentially closed. Let $V=K^1$, then $B_1(V)=K^3$. Let $b_1,b_2\in K$ be such that $\delta_2(b_1)=0$ and $\delta_1(b_2)\neq 0$. If we let 
$$W:=K\times\{b_1\}\times\{b_2\},$$
then clearly $W\subset B_1(V)$ projects dominantly onto $V$. In this case we have that $\{v\in V:\nabla(v)\in W\}=\emptyset$. Indeed, if $v\in V$ is such that $\nabla(v)\in W$, then $\delta_1(v)=b_1$ and $\delta_2(v)=b_2$ and so (by the assumptions on $b_1$ and $b_2$)
$$0=\delta_2(\delta_1(v))=\delta_1(\delta_2(v))\neq 0.$$
\end{exam} 

Let us explain the root of the problem. Consider a differential field $(F,\delta_1,\ldots,\delta_m)$ and a tuple $(a,b_1,\ldots,b_m)$ of some field extension of $F$. We would like to know if there exists a differential field extension $(M,D_1,\ldots,D_m)$ of $F$ such that $D_ia=b_i$ for $i=1,\ldots,m$. A necessary condition for such an extension to exist is that
\begin{equation}\label{niceq}
\frac{df}{dx}(a)\, b_i+f^{\delta_i}(a)=0, \quad i=1,\ldots,m
\end{equation}
for all $f \in F[x]$ vanishing at $a$. In fact, by Theorem 5.1 in Chap. 7 of \cite{LangBook}, if the above equations are satisfied we obtain derivations $D_i:F(a)\to F(a,b_i)$ extending $\delta_i$ such that $D_i(a)=b_i$. So, let us assume that the tuple $(a,\bar b)=(a,b_1,\ldots,b_m)$ satisfies (\ref{niceq}). 

Now, in order to build the desired differential field extension one needs to find a tuple $(c_{i,j})_{1\leq i,j\leq m}$ satisfying
\begin{equation}\label{niceq2}
\frac{\partial f}{\partial x}(a,\bar b)b_i + \sum_{j=1}^m\frac{\partial f}{\partial y_j}(a,\bar b)c_{i,j}+ f^{\delta_i}(a,\bar b)=0, \quad i=1,\ldots,m
\end{equation}
for all $f\in F[x,y_1,\ldots,y_m]$ vanishing at $(a,\bar b)$, and, since we want the derivations to commute, we also require the following \emph{integrability equations}
$$c_{i,j} =c_{j,i} \quad 1\leq i,j\leq m.$$
Indeed, if there were such a tuple of $c_{i,j}$'s, by the same token as above, we would obtain derivations $D_i:F(a,\bar b)\to F(a,\bar b, (c_{i,j})_{1\leq i,j\leq m})$ extending $\delta_i$ such that $D_i(a)=b_i$ and $D_i(b_j)=c_{i,j}=c_{j,i}=D_j(b_i)$.

In the ordinary case, there are no integrability equations and one can show, rather easily, that there exists such a $c_{1,1}$ satisfying (\ref{niceq2}). Hence, in this case, by only assuming (\ref{niceq}), this process yields the desired extension. However, in the case of finitely many commuting derivations, the complete system (i.e., (\ref{niceq2}) together with the integrability equations) might be inconsistent and so no such differential field extension would exist. The issue is that the integrability equations are not in general implied by (\ref{niceq}). Therefore, the additional relations imposed by the commutativity of the derivations must be taken into account in order to prove a proper analogue of Fact \ref{ordi}. To do this we make use of results from \cite{Pierce2014fields} on the axioms of differentially closed fields with several commuting derivations. We first recall some of the terminology of that paper.

Let $(F,\delta_1,\ldots,\delta_m)$ be a differential field and fix $n\in\mathbb N$. We will consider the partial order $\leq$ defined on $\mathbb N^m\times n$ by $(\xi, i) \leq (\zeta, j)$ if and only if $i=j$ and $\xi$ is less than or equal to $\zeta$ in the product order of $\mathbb N^{m}$. Note that if $x=(x_1,\ldots,x_{n})$ are differential indeterminates and if we identify $(\xi,i)$ with $\delta^\xi x_{i+1}:=\delta_m^{\xi_m}\cdots\delta_1^{\xi_1}x_{i+1}$, then $\leq$ induces an order on the algebraic indeterminates given by $\delta^\xi x_i\leq \delta^\zeta x_j$ iff $\delta^\zeta x_j$ is a derivative of $\delta^\xi x_i$ (in particular this implies that $i=j$).

Recall that for $\xi\in\mathbb N^m$ we let $|\xi|:=\xi_1+\cdots+\xi_m$. We will also consider the total order $\unlhd$ on $\mathbb N^m\times n$ defined by $(\xi,i)\unlhd (\zeta,j)$ if
$$(|\xi|, i,\xi_m,\ldots,\xi_1) \text{ is less than or equal } (|\zeta|,j,\zeta_m,\ldots,\zeta_1)$$
in the lexicographic order. Then $(\mathbb N^m\times n,\unlhd)$ has order type $(\omega,\in)$, and it induces the canonical orderly ranking on the algebraic indeterminates.

If $1\leq k\leq m$, we will let $\bf k$ denote $(0,\dots,1,\dots, 0)\in \mathbb N^m$ where the $1$ is in the $k^{\operatorname{th}}$-coordinate. Recall that for $\ell\in\mathbb N$ we let $\Gamma(\ell):=\{\xi\in\mathbb N^m: |\xi|\leq \ell\}$. Let $L$ be a finitely generated field extension of $F$ of the form
\begin{equation}\label{exten}
L=F(a_i^\xi: i<n, \xi\in\Gamma(s)),
\end{equation}
for some positive integer $s$. Note that $(a_i^{\xi}:i<n, \xi\in \Gamma(s))$ is simply a way to enumerate (label) the generators of the extension. Let $L':=F(a_i^\xi: i<n, \xi\in\Gamma(s-1))$. It is said that $L$ satisfies the \emph{differential condition} if for each $k=1,\dots,m$ there is a derivation $D_k: L'\to L$ extending $\delta_k$ such that $D_k a_i^\xi=a_i^{\xi+\bf k}$ for all generators $a_i^\xi$ of $L'$. 

A differential field extension $(M,D_1,\dots,D_m)$ of $(F,\delta_1,\ldots,\delta_m)$ is said to be \emph{compatible} with $L$ (as given in (\ref{exten})) if $L\leq M$ and $D_ka_i^\xi=a_i^{\xi +\bf k}$ for all $k=1,\dots,m$ and $a_i^{\xi}$ with $|\xi|\leq s-1$. A generator $a_i^\xi\in L$ is said to be a \emph{leader} if $a_i^\xi$ is algebraic over $F(a_j^\zeta:(\zeta,j)\triangleleft (\xi,i))$. A leader $a_i^\xi$ is said to be \emph{minimal} if it is a minimal element in the set of leaders with respect to the order $\leq$, in other words, if there is no leader $a_j^\zeta$ such that $(\zeta,j)<(\xi,i)$. 

We now recall some of the results from \cite{Pierce2014fields}. Note that the following gives sufficient conditions for the existence of the differential field extension discussed after Example \ref{example1}.

\begin{fact}\cite[Theorem 4.3]{Pierce2014fields}\label{Pierce1}
Assume that $F(a_i^\xi: i<n, \xi\in\Gamma(2r))$ satisfies the differential condition for some positive integer $r$. If for all minimal leaders $a_i^\xi$ we have that $|\xi|\leq r$, then $(F,\delta_1,\ldots,\delta_m)$ has a differential field extension compatible with $F(a_i^\xi: i<n, \xi\in\Gamma(2r-1))$.
\end{fact}

The following will be the main ingredient of our analogue of Fact \ref{ordi} in the case of fields with several commuting derivations (see Proposition \ref{goodfact}).

\begin{fact}\cite[Theorem 4.10]{Pierce2014fields}\label{Pierce2}
Let $r$ be a positive integer. There is an integer $s\geq r$, that depends only on $m$, $n$ and $r$, such that if $F(a_i^\xi: i<n, \xi\in\Gamma(s))$ satisfies the differential condition, then $(F,\delta_1,\ldots,\delta_m)$ has a differential field extension that is compatible with $F(a_i^\xi: i<n, \xi\in\Gamma(r))$.
\end{fact}

We aim to give an algorithm to express $s$ in terms of $m,n$ and $r$. To do this, let us recall the proof of Fact~\ref{Pierce2} where one uses the existence of certain bounds such as the following:

\begin{fact}\label{anti1}
Given an increasing sequence $(a_i:i\in \mathbb N)$ of positive integers, there is $t\in \mathbb N$ (depending only on $m$, $n$ and the $a_i$'s) such that  any chain $S_0\subseteq S_1\subseteq \cdots \subseteq S_t$ of antichains of $(\mathbb N^m\times n,\leq)$ with $S_k\subseteq\{(\xi, i):|\xi|\leq a_k\}$ is not strictly increasing.
\end{fact}

The proof of Fact~\ref{Pierce2} (as it appears in \cite{Pierce2014fields}) goes as follows. Applying Fact~\ref{anti1} with the sequence $(2^i r:i\in \mathbb N)$, we know that there is some value $t$ (depending only on $m$, $n$ and $r$) such that any chain $S_0\subseteq S_1\subseteq\cdots\subseteq S_t$ of antichains of $(\mathbb N^m\times n, \leq)$ such that $S_k\subseteq\{(\xi, i):|\xi|\leq 2^k r\}$ is not strictly increasing. 

We claim that taking $s=2^t r$ in the proof of \ref{Pierce2} does the job. Suppose $F(a_i^\xi:i<n, \xi\in\Gamma(2^t r))$ satisfies the differential condition. For each $u\leq t$, let $F_u=F(a_i^\xi: i<n, \xi\in\Gamma(2^u r))$ and $S_u$ be the set of minimal leaders of $F_u$. Note that $S_u\subseteq \{(\xi,i):|\xi|\leq 2^u r\}$. Then, by the choice of $t$, $S_u=S_{u+1}$ for some $u<t$, and so $F_{u+1}$ satisfies the hypothesis of Fact~\ref{Pierce1}. Hence, $(F,\delta_1,\ldots,\delta_m)$ has a differential field extension compatible with $F(a_i^\xi: i<n, \xi\in\Gamma(2^{u+1} r-1))$, and therefore also compatible with $F(a_i^\xi: i<n,\xi\in\Gamma(r))$, as desired. 

We note that in order to give an effective method to find $s$ it suffices to find such a method for $t$. In what follows we provide an algorithm to compute $t$ (finding such an effective algorithm is a problem that was originally studied by Seidenberg \cite{Seid} in the 1950's). It is worth mentioning that parts of our algorithm are more or less implicit in the proof of Fact \ref{anti1} given by Pierce. Moreover, the existence and recursive algorithms to compute similar bounds have been established (since the 1980's) using general versions of Dickson's lemma, see for example \cite{figueira2011}, \cite{mcaloon1984petri} or \cite{Socias}. The reason we present here an explicit algorithm is to keep the paper as self-contained as possible and to justify the effectiveness of our bounds.

\subsection{Algorithm to compute $t$ of Fact \ref{anti1}}

The construction of $t=t(m,n,(a_i:i\geq 0))$ is recursive. We will do the construction while proving that such construction works. The base cases are $m=1,2$. In the case $m=1$, it is clear that for all $n$ and $(a_i:i\geq 0)$, $t(1,n,(a_i:i\geq 0))=n+1$. When $m=2$, we first prove

\begin{lem}\label{case1}
If $S$ is an antichain of $\mathbb N^2$ with respect to the product ordering and $\sigma\in S$, then $|S|\leq |\sigma|+1$.
\end{lem} 

\begin{proof} 
To prove the lemma it suffices to show that we can embed $S$ in $S_\sigma=\{(s_1,s_2)\in \mathbb N^2: s_1+s_2=|\sigma|\}$, since the latter has cardinality $|\sigma|+1$. Let $\sigma=(\sigma_1,\sigma_2)$ and consider the map $f:S\to S_\sigma$ given by
\begin{equation}
\bar s=(s_1,s_2)\mapsto 
\left\{
\begin{array}{cc}
\sigma & \text{ if } \bar s=\sigma\\
(s_1,|\sigma|-s_1) & \;\;\text{ if } s_1<\sigma_1\\
(|\sigma|-s_2,s_2) &  \;\;\text{ if } s_2< \sigma_2\\
\end{array}
\right.
\end{equation}
It is clear that $f(\bar s)\in S_\sigma$ and that $f$ is injective. 
\end{proof} 

\begin{lem}\label{intwo}
The following recursive definition gives the value of $t$ for $m=2$: 
$$t(2,n,(a_i:i\geq 0))=b_{n}+1,$$
where $b_0=0$ and $b_{i+1}= a_{b_{i}+1}+b_i+1$ for $i=0,\dots,n-1$. 
\end{lem} 
\begin{proof} 
Suppose there is a strictly increasing chain $S_0\subset\cdots\subset S_{r}$ of antichains of $(\mathbb N^2\times n,\leq)$ with $S_k\subseteq\{(\xi, i):|\xi|\leq a_k\}$. To prove the lemma it suffices to show that $|S_r|\leq b_n$. We proceed by induction on $n$. The base case $n=1$ follows from Lemma \ref{case1}. Let $n>1$. Let $S_k^{i}=\{\sigma\in \mathbb N^2: (\sigma,i)\in S_k\}$ for $i< n$. We now claim that, after reordering the elements of $n$ if necessary, if $S_k^i\neq S_{k+1}^{i}$ for some $0<i<n$, then for some $p\leq k$ we have $S_p^{i-1}\neq S_{p+1}^{i-1}$. Indeed this can be accomplished as follows. Let $k_0=0$. We have that $S_{k_0}^{i_0}\neq S_{k_0+1}^{i_0}$ for some $i_0<n$. Reorder the elements of $n=\{0,1,\dots,n-1\}$ in such a way that $i_0=0$. Now let $k_1$ be the smallest such that $S_{k_1}^{i_1}\neq S_{k_1+1}^{i_1}$ for some $i_1> 0$. Reorder $\{1,2,\dots,n-1\}$ in such a way that $i_1=1$. Now let $k_2$ be the smallest such that $S_{k_2}^{i_2}\neq S_{k_2+1}^{i_2}$ for some $i_2>1$. Reorder $\{2,3,\dots,n-1\}$ in such a way that $i_2=2$. We continue this procedure for $j=3,\dots,n-1$ by letting $k_j$ be the smallest such that $S_{k_j}^{i_j}\neq S_{k_j+1}^{i_j}$ for some  $i_j>j-1$, and reordering $\{j,j+1,\dots,n-1\}$ in such a way that $i_j=j$. Note that $k_0\leq k_1\leq\cdots\leq k_{n-1}$. After this reordering of $n$ we obtain that if $S_k^i\neq S_{k+1}^i$ for some $0<i<n$, then $S_{k_i}^i\neq S_{k_i+1}^i$, and consequently $S_{k_{i-1}}^{i-1}\neq S_{k_{i-1}+1}^{i-1}$. Since $k_{i-1}\leq k_i$, we have proven our claim.

Now, by induction, we have that for each $i<n$ there is $k_i\leq b_{i}+1$ such that $S_{k_i}^i\neq \emptyset$, say $\sigma_i\in S_{k_i}^i$. Since the sequence of $a_i$'s is increasing, we have that $|\sigma_i|\leq a_{b_{i}+1}$. By Lemma \ref{case1}, $|S_r^{i}|\leq a_{b_{i}+1}+1$ for all $i<n$. Thus, since $S_r=\cup_{i<n}S_r^i$, we have
$$|S_r|\leq\sum_{i<n}(a_{b_{i}+1}+1)=a_{b_{n-1}+1}+b_{n-1}+1=b_n,$$
as desired. 
\end{proof} 

\noindent {\bf Recursive construction of $t=t(m,n,(a_i:i\geq 0))$ for $m>2$:} 

First we consider the case $t(m,1,(a_i:i\geq 0))$. We assume that we have recursively constructed $t(m-1,n,(d_i:i\geq 0))$ for arbitrary $n$ and sequence $(d_i)$. Thus, for each sequence $(d_i:i\geq 0)$, we have (recursively) defined $g:\m N\to \m N$ such that $g(k)=t(m-1,n,(d_i:i\geq k))$. We call such a $g$ the \emph{bound function} associated to $(d_i:i\geq 0)$. Let $f_1$ be the bound function associated to the given sequence $(a_i:i\geq 0)$; that is, $f_1(k) = t( m-1, n , (a_i:i\geq k))$.

Now suppose that we have an increasing chain $S_0\subset S_1\cdots\subset S_r$. We need to find an upper bound for $r$. Let $\xi\in S_1$ and $a:=a_1$. For each $1\leq i\leq m$ and $j\in\m N$, let $S_k^{i,j}=\{(\zeta,l)\in S_k:\zeta_i=j\}$. By the antichain assumption, if $\zeta\in S_k$ then there exists $1\leq i\leq m$ such that $\zeta_i\leq \xi_i$, and consequently $\zeta\in \bigcup_{j\leq \xi_i}S_k^{i,j}$. Thus,
$$S_k=\bigcup_{1\leq i\leq m,j\leq \xi_i}S_k^{i,j}.$$
Moreover, since $|\xi|\leq a$, we have $\xi_i\leq a$ and $\bigcup_{j\leq \xi_i}S_k^{i,j}\subseteq \bigcup_{j\leq a}S^{i,j}$ for all $1\leq i\leq m$. It then follows that
\begin{equation}\label{eq23}
S_k=\bigcup_{1\leq i\leq m,j\leq \xi_i}S_k^{i,j}=\bigcup_{1\leq i\leq m,j\leq a}S_k^{i,j},
\end{equation}
which is a union of at most $m(a+1)$-many sets. Thus we may write 
$$S_k=\bigcup_{i=0}^{p:=m(a+1)-1}S_k^i,$$
where each $S_k^i$ is one of the $S_k^{i,j}$ appearing in the right hand side of (\ref{eq23}).

Since $S_0\subset S_1$, we have that $S_0^{\ell_0}\subset S_1^{\ell_0}$ for some $1\leq \ell_0\leq p$. Now, the sequence $S_0^{\ell_0}\subseteq\cdots \subseteq S_{f_1(0)}^{\ell_0}$ can not be strictly increasing and so we must have that $S_0^{\ell_1}\subset S_{f_1(0)}^{\ell_1}$ for some $1\leq \ell_1\leq p$ with $\ell_1\neq \ell_0$. Letting $b^1_0=0$, $b^1_{l+1}=f_1(b^1_l)+b^1_l$ and $f_{2}$ be the bound function associated to the sequence $(a_{b^1_l}:l\geq 0)$, we see that there must be some $1\leq \ell_2\leq p$ with $\ell_2\notin \{\ell_0,\ell_1\}$ such that $S_0^{\ell_2}\subset S_{b^1_{f_2(0)}}^{\ell_2}$. Continuing in this fashion for $i=2,\dots,p-1$, if we let $b^i_0=0$, $b^i_{l+1}=b^{i-1}_{f_{i}(b^i_l)}+b^i_l$ and $f_{i+1}$ be the bound function associated to the sequence $(a_{b^i_l}:l\geq 0)$, we see that there must be $1\leq \ell_{i+1}\leq p$ with $\ell_{i+1}\notin\{\ell_0,\dots,\ell_i\}$ such that $S_0^{\ell_{i+1}}\subset S_{b^i_{f_{i+1}(0)}}^{\ell_{i+1}}$. Once we get to $i=p-1$, we obtain $f_{p}$ and $1\leq \ell_p\leq p$ with $\ell_{p}\notin\{\ell_0,\dots,\ell_{p-1}\}$ such that $S_0^{\ell_p}\subset S_{b^{p-1}_{f_{p}(0)}}^{\ell_p}$. Finally, letting $b_0=0$, $b_{l+1}=b^{p-1}_{f_{p}(b_l)}+b_l$ and $f$ be the bound function associated to the sequence $(a_{b_l}:l\geq 0)$, we see that 
$r<b_{f(0)}$. Thus, we let $t(m,1,(a_i:i\geq 0)):=b_{f(0)}$.

Finally, we consider the case $t(m,n,(a_i:i\geq 0))$ for $m>2$ and $n>1$. By induction, we assume that we have recursively defined bound functions for $m$ and $n'<n$. From now on, when we use the term bound function we mean we respect to the fixed $m$. Let $f$ be the bound function associated to $n'=1$ and the given sequence $(a_i:i\geq 0)$; that is, $f(k)=t(m,1,(a_i:i\geq k))$. Suppose we have an increasing chain $S_0\subset S_1\subset \cdots\subset S_r$. Let 
$$S_k'=\{(\xi,i)\in S_k: \,i=n-1\} \quad \text{ and }\quad S_k''=\{(\xi,i)\in S_k: \, i<n-1\}.$$
The sequence $S'_0\subseteq\cdots\subseteq S'_{f(0)}$ can not be strictly increasing and so $S''_0\subset S''_{f(0)}$. Leting $b_0=0$, $b_{l+1}=f(b_l)+b_l$, and $g$ be the bound function associated to $n'=n-1$ and the sequence $(a_{b_l}: l\geq 0)$, we see that $r< b_{g(0)}=: t(m,n,(a_i:i\geq 0))$. This completes our algorithm.

\

To conclude this section let us consider again the situation of Fact \ref{Pierce2}. Let $r$ be a positive integer and consider the sequence $(2^i r:i\in \mathbb N)$. Let $t=t(m,n,r)$ be the bound associated to this sequence given by the above algorithm. By the argument following Fact \ref{anti1} (and noting that in the ordinary case one can simply take $s=r$), the value of $s$ in Fact \ref{Pierce2} can be taken to be $T_r^{m,n}$ where
\begin{equation}\label{defT}
\left\{
\begin{array}{cc}
T_r^{1,n}= r \qquad\quad&  \\
T_r^{m,n}=2^{t(m,n,r)} r, &\; \text{ when }m>1.
\end{array}
\right.
\end{equation}

Here are some calculations of $T_1^{2,n}$ for $n=1,2,3$:
\begin{enumerate}
\item If $n=1$, then $t=b_1+1=4$ and so $T_1^{2,1}=2^4=16$.
\item If $n=2$, then $t=b_2+1=2^4+5=21$ and so $T_1^{2,2}=2^{2^4+5}=2097152$.
\item If $n=3$, then $t=b_3+1=2^{2^4+5}+2^4+6=2097174$, and so $T_1^{2,3}=2^{2^{2^4+5}+2^4+6}$.
\end{enumerate}

\section{Uniform Bounding} \label{count}

In this final section we prove our main result: an effective version of uniform bounding for partial differential fields.  We first prove an analogue of  Fact \ref{ordi} for partial differential fields. To do this we will use the results of the previous section and so, for the rest of this section, we fix $T=T_1^{m,n}$ (where the latter was defined in (\ref{defT})). Note that $T$ only depends on $m$ and $n$. As in Section \ref{prolongingthemagic}, $(K,\delta_1,\ldots,\delta_m)$ denotes a differentially closed field of characteristic with commuting derivations. 

\begin{prop}\label{goodfact}
Let $V$ be an irreducible subvariety of $K^n$ and $W$ a subvariety of $K^{n(m+1)}$. Suppose $X$ is an irreducible subvariety of $B_{T-1}(V)$ such that 
$$B_{T-1}(W)\cap B_1(X)\subseteq B_T(V)$$ 
projects dominantly onto $X$. If $V'$ is a subvariety of $V$ which does not contain the projection of $X$ in $V$, then there is $v\in V\setminus V'$ such that $\nabla(v)\in W$.
\end{prop}

\begin{rem} 
Note that in the case when $m=1$, and so $T=1$, the assumptions of the proposition reduce to \emph{$W\cap B_1(X)$ projects dominantly onto $X$}. Thus, in this case, we recover Fact \ref{ordi}. 
 \end{rem}

\begin{proof}
We take $\mathbb U$ to be a universal extension of $K$; that is, $\mathbb U$ is an algebraically closed field extension of $K$ that contains a generic point for each affine algebraic variety defined over $K$. For any closed subvariety $Z$ of $K^n$ and any intermediate field $K\leq F\leq \mathbb U$, we let $Z(F)$ be the zero set in $F$ of the ideal of polynomials over $K$ vanishing on $Z$. Now, let $Y$ be an irreducible component of $B_{T-1}(W)\cap B_1(X)\subseteq B_T(V)$ that projects dominantly onto $X$. Let $(a_{i}^\xi:  i<n, \xi\in \Gamma(T))$ be a generic point of $Y(\mathbb U)$ over $K$. Then $\hat a:=(a_{i}^\xi: i<n, \xi\in\Gamma(T-1))$ is a generic point of $X(\mathbb U)$ over $K$. Also, note that since the projection of $X$ in $V$ is not contained in $V'$, the point $(a_{i}^{\bf 0}:i<n)$ is in $V(\mathbb U)\setminus V'(\mathbb U)$. We claim that $L:=K(a_{i}^\xi: i<n,\xi\in\Gamma(T))< \mathbb U$ satisfies the differential condition. Fix $1\leq k\leq m$. To prove that $L$ satisfies the differential condition it suffices to show that there is a derivation $D_k:L' \to L$, where $L'=K(\hat a)$, extending $\delta_k$ such that $D_ka_i^\xi=a_i^{\xi+{\bf k}}$ for all $a_i^{\xi}\in L'$. By the standard argument for extending a single derivation (see \cite[Chapter~7, Theorem~5.1]{LangBook}, for instance), it suffices to show that if $f$ is a polynomial over $K$ in variables $(x_{i}^\xi: i<n, \xi\in\Gamma(T-1))$ and $f(\hat a)=0$, then 
\begin{equation}\label{news}
\sum_{i<n,\xi\in\Gamma(T-1)}\frac{\partial f}{\partial x_{i}^\xi}(\hat a)\cdot a_{i}^{\xi+{\bf k}}+f^{D_k}(\bar a)=0.
\end{equation}
Since $\hat a$ is a generic point of $X(\mathbb U)$, the above equation is one of the equations defining the first prolongation $\tau_1V(\mathbb U)$ (as discussed in \S2) evaluated at $(a_{i}^\xi: i<n, \xi\in\Gamma(T))$. We have seen that $B_1(X)(\mathbb U)\subseteq \tau_1V(\mathbb U)$, and hence, since
$$(a_{i}^\xi: i<n, \xi\in\Gamma(T))\in B_{1}(X)(\mathbb U),$$
we  have that equation (\ref{news}) indeed holds. This shows that $L$ satisfies the differential condition. By Fact \ref{Pierce2} and the choice of $T$, there is a differential field extension $(M,\Delta)$ of $(K,\Delta)$ that is compatible with $K(a_{i}^\xi: i<n,\xi\in\Gamma(1))$. By universality of $\mathbb U$ over $K$, we have that $M< \mathbb U$. Since $(a_{i}^\xi:  i<n, \xi\in \Gamma(T))\in B_{T-1}(W)(\mathbb U)$, we have that $(a_i^\xi:i<n,\xi\in\Gamma(1))\in W(\mathbb U)$. Thus, in the structure $(M,\Delta)$, we have that $\nabla(a_{i}^{\bf 0}:i<n)\in W(M)$, and above we had already seen that $(a_{i}^{\bf 0}:i<n)$ is in $V(M)\setminus V'(M)$. Using the fact that $(K,\Delta)$ is differentially closed, one can now find a point in $K$ with the desired properties.
\end{proof}

To prove uniform bounding (for partial differential polynomial equations), let us first consider the case of systems of first-order differential equations of the form (in Corollary \ref{finalcor} we prove the general case):
\begin{align*}
 p_1(x) &=0 
\\ &\vdots
\\ p_r(x)&=0
\\ q_1(x,\delta_1 x,\ldots,\delta_m x)&=0
\\ &\vdots
\\ q_s(x,\delta_1 x,\ldots,\delta_m x)&=0
\end{align*}
where $p_i$ and $q_i$ are polynomials over $K$ in the variables $x=(x_1,\ldots,x_{n})$ and $(x^\xi)_{\xi\in \Gamma(1)}$, respectively. Let $Z$ be the set of solutions in $K$ of this system and assume that it is finite. Our goal is to give an upper bound on the cardinality of $Z$ in terms of $m$, $n$ and the degrees of the $p_i$'s and the $q_i$'s. 

Let $V$ be the closed subvariety of $K^n$ defined by the $p_i$'s and $W$ be the closed subvariety of $K^{n(m+1)}$ defined by the $q_i$'s, then the set of solutions is given by $Z=\{v\in V:\nabla(v)\in W\}$. Essentially what we will do is provide an algorithm which will compute an upper bound for the size of $Z$. The termination of the algorithm follows by noetherianity of the Zariski topology. The upper bound will follow from keeping track of degrees and dimensions of certain algebraic varieties. Here is an informal description of the algorithm: 

\begin{enumerate} 
\item Set $X$ to be the Zariski closure of $\pi(B_T(V)\cap B_{T-1}(W))$ where $\pi$ is the canonical projection $B_1(B_{T-1}(V))\to B_{T-1}(V)$. 
\item Is the projection of $X$ in $V$ a finite set?
\begin{enumerate}
\item If yes, we stop.
\item If no, replace $X$ with the Zariski closure of $\pi(B_{T-1}(W)\cap B_1(X))$ and go back to the beginning of Step 2 (as we will see in the proof of Theorem \ref{maint}, this step will decrease the dimension of the irreducible components of $X$ whose projection in $V$ is infinite). 
\end{enumerate}
\end{enumerate} 

\begin{thm}\label{maint}
Let $V$ be a subvariety of $K^n$ of dimension $d$ and $W$ a subvariety of $K^{(m+1)n}$. If $Z=\{v\in V: \nabla(v)\in W\}$ is finite, then 
$$|Z|\leq (\deg V)^{\alpha_T \cdot(m+1)^{d\alpha_{T-1}-1}}\;(\deg W)^{\alpha_{T-1}\cdot\frac{(m+1)^{d\alpha_{T-1}}-1}{m}}.$$
\end{thm}
\begin{proof}
First we show that if the statement holds for all irreducible $V$ of dimension $\leq d$, then it holds for arbitrary $V$ of dimension $d$. So, assume the statement holds for all irreducible varieties of dimension $\leq d$, and suppose $V$ has decomposition $V_1\cup\dots\cup V_s$ and dimension $d$. Let $d_i=\dim V_i$ and $Z_i=V_i\cap Z$. Then $Z_i=\{v\in V_i: \nabla(v)\in W\}$, and so, since $V_i$ is irreducible, our assumption implies
$$|Z_i|\leq (\deg V_i)^{\alpha_T \cdot(m+1)^{d_i\alpha_{T-1}-1}}\;(\deg W)^{\alpha_{T-1}\cdot\frac{(m+1)^{d_i\alpha_{T-1}}-1}{m}}.$$
Using that $Z=Z_1\cup\dots\cup Z_s$, we get
\begin{equation}\label{subit1}
|Z|\leq \sum_{i=1}^s |Z_i| \leq \sum_{i=1}^s (\deg V_i)^{\alpha_T \cdot(m+1)^{d_i\alpha_{T-1}-1}}\;(\deg W)^{\alpha_{T-1}\cdot\frac{(m+1)^{d_i\alpha_{T-1}}-1}{m}}.
\end{equation}
Recalling that by definition $\deg V=\sum_{i=1}^s \deg V_i$, we have
$$\sum_{i=1}^s (\deg V_i)^{\alpha_T \cdot(m+1)^{d\alpha_{T-1}-1}}\leq \left(\sum_{i=1}^s \deg V_i\right)^{\alpha_T \cdot(m+1)^{d\alpha_{T-1}-1}}=(\deg V)^{\alpha_T \cdot(m+1)^{d\alpha_{T-1}-1}}.$$
Using this inequality, together with (\ref{subit1}) and the fact that $d_i\leq d$, we get
$$|Z|\leq (\deg V)^{\alpha_T \cdot(m+1)^{d\alpha_{T-1}-1}}\;(\deg W)^{\alpha_{T-1}\cdot\frac{(m+1)^{d\alpha_{T-1}}-1}{m}},$$
as desired. Thus, it suffices to prove the statement under the assumption that $V$ is irreducible. 

If $d=0$ we have that $|Z|\leq |V|= \deg V\leq (\deg V)^{\alpha_T (m+1)^{-1}}$, and so the upper bound holds (note that $\alpha_T(m+1)^{-1}\geq 1$ ). We now assume $d\geq 1$. 

Let $X_0$ be the Zariski closure of $\pi(B_T(V)\cap B_{T-1}(W))$ where $\pi:B_1(B_{T-1}(V))\to B_{T-1}(V)$ is the canonial projection. Define recursively
$$X_s:=\text{ Zariski-closure of }\pi(B_{T-1}(W)\cap B_1(X_{s-1}))$$
for $s\geq 1$. It is easy to check, by induction on $s$ and using Fact~\ref{Bezout}, that $\nabla_{T-1}(Z)\subseteq X_s$ and 
\begin{equation}\label{onx}
\deg (X_s)\leq (\deg V)^{\alpha_T(m+1)^s}\;(\deg W)^{\alpha_{T-1}\sum_{i=0}^s(m+1)^i}.
\end{equation}
Let us explain how the induction step is performed to prove (\ref{onx}): Assume the inequality is true for $s-1$. Using Fact~\ref{Bezout} and the definition of $X_s$ we get $\deg X_s\leq \deg B_{T-1}(W)\cdot \deg B_1(X_{s-1})$. Recall from \S2 that $\deg B_{T-1}(W)\leq (\deg W)^{\alpha_{T-1}}$ and $\deg B_1(X_{s-1})\leq (\deg X_{s-1})^{m+1}$. These inequalities (and the inductive hypothesis) imply
\begin{align*}
\deg X_s\leq & \; (\deg W)^{\alpha_{T-1}}\; (\deg V)^{(m+1)\alpha_T(m+1)^{s-1}}\; (\deg W)^{(m+1)\alpha_{T-1}\sum_{i=0}^{s-1}(m+1)^i} \\
= & \; (\deg V)^{\alpha_T(m+1)^s}\;(\deg W)^{\alpha_{T-1}\sum_{i=0}^s(m+1)^i}.
\end{align*}
Thus we obtain (\ref{onx}).

We now note that 
\begin{equation}\label{ondim}
\dim X_0<d\alpha_{T-1}
\end{equation}
Indeed, if this were not the case, we would have that $X_0=B_{T-1}(V)$ (here one uses that fact that, since $V$ is irreducible, $B_{T-1}(V)$ is irreducible) and so Proposition \ref{goodfact} would imply that $|Z|$ is infinite.

\vspace{.05in}
\noindent {\bf Claim 1.} There is $s\leq d\alpha_{T-1}-1$ such that the projection $X_s\to V$ is a finite set.

\vspace{.03in}
\noindent {\it Proof of Claim.} Towards a contradiction suppose there is no such an $s$. Then, for each $s< d\alpha_{T-1}-1$, the projections of $X_s$ and $X_{s+1}$ in $V$ are infinite sets. Let $\hat X_s$ be the union of all the irreducible components of $X_s$ whose projection in $V$ is infinite. Similarly for $\hat X_{s+1}$. We now claim that $\dim \hat X_{s+1}< \dim \hat X_s$. Suppose this is not the case. Then there is an irreducible component $Y$ of $X_{s+1}$ whose projection in $V$ is infinite and such that $\dim Y= \dim \hat X_s$. Since $X_{s+1}\subseteq X_s$, we have that $Y$ is an irreducible component of $X_s$. Since $B_{T-1}(W)\cap B_1(X_s)$ projects dominantly onto $X_{s+1}$ and $Y$ is an irreducible component of $X_{s+1}$, $B_{T-1}(W)\cap B_1(X_s)\cap \pi^{-1}(Y)$ projects dominantly onto $Y$. By Lemma \ref{irrb}, $B_1(Y)$ is an irreducible component of $B_1(X_s)$, and so there is a nonempty open set $U$ of $Y$ such that $B_1(X)\cap \pi^{-1}(U)=B_1(Y)\cap\pi^{-1}(U)$ ($U$ can be obtained by removing from $Y$ all the other components of $X_s$). Hence, $B_{T-1}(W)\cap B_1(Y)$ projects dominantly onto $Y$. Then, by Proposition \ref{goodfact}, $|Z|$ would be infinite. Hence, no such $Y$ can exist. 

The above discussion shows that $\dim (\hat X_s) \leq \dim (\hat X_0) -s$ for all $s\leq d\alpha_{T-1}$. Hence, 
$$\dim (\hat X_{d\alpha_{T-1}-1} )\leq \dim (\hat X_0 )- d\alpha_{T-1}+1\leq 0,$$
where the last inequality uses (\ref{ondim}) (and $\dim (\hat X_0) \leq \dim (X_0)$). This shows that the irreducible components of $X_{d\alpha_{T-1}-1}$ whose projection in $V$ is infinite have dimension zero. But this is impossible and so we have proven the claim.

Let $s$ be as in Claim 1. Then the projection $X_s\to V$ is a finite set. Since $\nabla_{T-1}(Z)\subseteq X_s$, Fact \ref{Bezout} yields $|Z|\leq \deg X_s$. Using (\ref{onx}) and $s\leq d\alpha_{T-1}-1$ we get
\begin{equation}\label{sub}
|Z|\leq \deg X_s \leq (\deg V)^{\alpha_T (m+1)^{d\alpha_{T-1}-1}}\;(\deg W)^{\alpha_{T-1} \sum_{i=0}^{d\alpha_{T-1}-1}(m+1)^i}.
\end{equation}
Recall that for any integer $a>1$ we have $\sum_{i=0}^k a^i=\frac{a^{k+1}-1}{a-1}$, and so $$\sum_{i=0}^{d\alpha_{T-1}-1}(m+1)^i=\frac{(m+1)^{d\alpha_{T-1}}-1}{m}.$$ 
The result follows by subbing this in the right hand side of (\ref{sub}).
\end{proof}

\begin{rem}\label{imprem} \
\begin{enumerate}
\item The proof actually yields an upper bound for the degree of the Zariski closure of $Z$, even when $Z$ is not finite. In fact, a slightly more detailed conclusion is given in Corollary \ref{posdim} below. 
\item We can now give an upper bound in terms of the degrees of the $p_1,\ldots,p_r$ and the $q_1,\ldots,q_s$. Suppose their degree is bounded by $D$. Then, by Bezout's inequality, $\deg V\leq D^{r}$ and $\deg W\leq D^s$. Since 
$$\alpha_{T-1}\cdot\frac{(m+1)^{d\alpha_{T-1}}-1}{m}\leq \alpha_T\cdot (m+1)^{d\alpha_{T-1}},$$ 
using the upper bound and that $\dim (V) \leq n$, we have
$$|Z|\leq D^{(r+s)\alpha_T(m+1)^{n\alpha_{T-1}}}.$$
\item If $S$ is a closed subvariety of $K^n$ and $Z'=\{v\in V\setminus S: \nabla(v)\in W\}$ is finite, then the above proof shows that the same upper bound holds for $|Z'|$ (i.e., the bound is independent of $S$). One could carry out a similar analysis in the case when $S$ is a subvariety of $K^{(m+1)n}$ and the Kolchin closed set is of the form $\{v\in V:\nabla(v)\in W\setminus S\}$; however, the arguments given here will not yield an upper bound independent of $S$ and so we would need a different strategy (for instance, this would require a small modification of Proposition \ref{goodfact} and an appropriate new algorithm). We do not explore such analysis here, as the case of $S\subseteq K^n$ is sufficient for our purposes (and for the possible applications that we pointed out in the introduction).
\end{enumerate}
\end{rem}

\begin{corr} \label{posdim} 
Let $V$ and $S$ be subvarieties of $K^n$ and $W$ a subvariety of $K^{(m+1)n}$. Let $d=\dim (V)$. If the components of the Zariski closure $\bar Z$ of $Z=\{v\in V\setminus S: \nabla(v)\in W\}$ have dimension greater than or equal to $d_0$, then 
$$\deg \bar Z \leq (\deg V)^{\alpha_T (m+1)^{(d-d_0)\alpha_{T-1}-1}}\;(\deg W)^{\alpha_{T-1}\cdot\frac{(m+1)^{(d-d_0)\alpha_{T-1}}-1}{m}}.$$
\end{corr} 
\begin{proof}
First we show that if the statement holds for all irreducible $V$ of dimension $\leq d$, then it holds for arbitrary $V$ of dimension $d$. So, assume the statement holds for all irreducible varieties of dimension $\leq d$, and suppose $V$ has decomposition $V_1\cup\dots\cup V_s$ and dimension $d$. Let $d_i=\dim V_i$ and $Z_i=V_i\cap Z$. Then $Z_i=\{v\in V_i\setminus S: \nabla(v)\in W\}$, and the irreducible components of $\bar Z_i$ (the Zariski closure of $Z_i$) are irreducible components of $\bar Z$. Thus, the components of $\bar Z_i$ have dimension greater than or equal to $d_0$, and so, since $V_i$ is irreducible, our assumption implies
$$\deg \bar Z_i \leq (\deg V_i)^{\alpha_T \cdot(m+1)^{(d_i-d_0)\alpha_{T-1}-1}}\;(\deg W)^{\alpha_{T-1}\cdot\frac{(m+1)^{(d_i-d_0)\alpha_{T-1}}-1}{m}}.$$
Using that $\bar Z=\bar Z_1\cup\dots\cup \bar Z_s$ and the fact that the components of $\bar Z_i$ are components of $\bar Z$, we have that $\deg \bar Z\leq \sum_{i=1}^s\deg \bar Z_i$. Putting this together with the above equation, we get
\begin{equation}\label{subit2}
\deg \bar Z  \leq \sum_{i=1}^s (\deg V_i)^{\alpha_T \cdot(m+1)^{(d_i-d_0)\alpha_{T-1}-1}}\;(\deg W)^{\alpha_{T-1}\cdot\frac{(m+1)^{(d_i-d_0)\alpha_{T-1}}-1}{m}}.
\end{equation}
Recalling that by definition $\deg V=\sum_{i=1}^s \deg V_i$, we have
\begin{align*}
\sum_{i=1}^s (\deg V_i)^{\alpha_T \cdot(m+1)^{(d-d_0)\alpha_{T-1}-1}} 
& \leq \left(\sum_{i=1}^s \deg V_i\right)^{\alpha_T \cdot(m+1)^{(d-d_0)\alpha_{T-1}-1}} \\
&=(\deg V)^{\alpha_T \cdot(m+1)^{(d-d_0)\alpha_{T-1}-1}}.
\end{align*}
Using this inequality, together with (\ref{subit2}) and the fact that $(d_i-d_0)\leq (d-d_0)$, we get
$$\deg \bar Z \leq (\deg V)^{\alpha_T \cdot(m+1)^{(d-d_0)\alpha_{T-1}-1}}\;(\deg W)^{\alpha_{T-1}\cdot\frac{(m+1)^{(d-d_0)\alpha_{T-1}}-1}{m}},$$
as desired. Thus, it suffices to prove the statement under the assumption that $V$ is irreducible. 

If $d=\dim (\bar Z)$, then, by irreducibility of $V$, we would have $\bar Z=V$, and so $\deg \bar Z= \deg V\leq (\deg V)^{\alpha_T (m+1)^{-1}}$. Thus the upper bound holds in this case (note that $\alpha_T(m+1)^{-1}\geq 1$). We now consider the remaining case $d>\dim (\bar Z)$.

Let $C_1,\dots, C_t$ be the irreducible components of $\bar Z$. For each $1\leq i\leq t$, let $Y_{i,0}$ be a component of the Zariski closure of $\pi(B_T(V)\cap B_{T-1}(W))$ such that the Zariski closure of the projection of $Y_{i,0}$ in $V$ contains $C_i$, where $\pi:B_1(B_{T-1}(V))\to B_{T-1}(V)$ is the canonial projection. We let 
$$X_0=Y_{1,0}\cup\cdots \cup Y_{t,0}.$$
Define recursively, for $s\geq 1$, $Y_{i,s}$ as a component of the Zariski closure of $\pi(B_{T-1}(W)\cap B_1(X_{s-1}))$ contained in $Y_{i,s-1}$ such that the Zariski closure of the projection of $Y_{i,s}$ in $V$ contains $C_i$. We let 
$$X_s=Y_{1,s}\cup\cdots\cup Y_{t,s}.$$
Induction on $s$ shows that $\nabla_{T-1}(Z)\subseteq X_s$ for all $s\geq 0$. Also, as we did in the proof of Theorem \ref{maint}, an induction on $s$ and using Fact~\ref{Bezout} yields
\begin{equation}\label{onx2}
\deg (X_s)\leq (\deg V)^{\alpha_T(m+1)^s}\;(\deg W)^{\alpha_{T-1}\sum_{i=0}^s(m+1)^i}.
\end{equation}

We also have that 
\begin{equation}\label{ondim2}
\dim (X_0)<d\alpha_{T-1}
\end{equation}
Indeed, if this were not the case, we would have that $X_0=B_{T-1}(V)$ (here one uses that fact that, since $V$ is irreducible, $B_{T-1}(V)$ is irreducible), and so Proposition \ref{goodfact} would imply that $\dim (\bar Z)=d$, which is impossible as we are assuming $d>\dim(\bar Z)$.

 \vspace{.05in}
\noindent {\bf Claim 1'.} There is $s\leq (d-d_0)\alpha_{T-1}-1$ such that the Zariski closure of the projection $X_s\to V$ equals $\bar Z$.

\vspace{.03in}
\noindent {\it Proof of Claim.} Towards a contradiction suppose there is no such an $s$. Then, for each $s\leq (d-d_0)\alpha_{T-1}-1$, the Zariski closure of the projection of $X_s$ in $V$ properly contains $\bar Z$; in other words, there is $Y_{i,s}$ such that $\dim (Y_{i,s})> \dim (\bar Z_i)$. We now claim that for each such $i$, $\dim (Y_{i,s+1})<\dim (Y_{i,s})$. Suppose this is not the case, i.e., $\dim(Y_{i,s+1})=\dim (Y_{i,s})$. By irreducibility, we get $Y_{i,s+1}=Y_{i,s}$. An argument like the one used in Claim 1 of Theorem \ref{maint} shows that $B_{T-1}(W) \cap B_1(Y_{i,s+1})$ projects dominantly onto $Y_{i,s+1}$. Now Proposition \ref{goodfact} would imply that $\dim (Y_{i,s})=\dim (Y_{i,s+1})=\dim (\bar Z_i)$, which is a contradiction. Hence, we must have $\dim (Y_{i,s+1})< \dim(Y_{i,s})$.

The above discussion shows that, for some $1\leq i\leq t$, $\dim (Y_{i,s}) \leq \dim (Y_{i,0}) -s$ for all $s\leq (d-d_0)\alpha_{T-1}$. Hence, 
$$\dim (Y_{i,(d-d_0)\alpha_{T-1}-1}) \leq \dim (Y_{i,0}) - (d-d_0)\alpha_{T-1}+1\leq d_0\alpha_{T-1},$$
where the last inequality uses (\ref{ondim2}) (and $\dim (Y_{i,0}) \leq \dim (X_0)$). Now, since the projection of $Y_{i,(d-d_0)\alpha_{T-1}}$ in $V$ contains $\bar Z_i$ and $\dim (\bar Z_i) \geq d_0$, we have that $\dim (Y_{i,(d-d_0)\alpha_{T-1}})\geq d_0\alpha_{T-1}$ (see Lemma \ref{goodcomp}). Thus, we obtain 
$$d_0\alpha_{T-1}\leq \dim (Y_{i,(d-d_0)\alpha_{T-1}}) \leq \dim (Y_{i,(d-d_0)\alpha_{T-1}-1}) \leq d_0\alpha_{T-1}.$$
Consequently, $\dim (Y_{i,(d-d_0)\alpha_{T-1}-1})=\dim (Y_{i,(d-d_0)\alpha_{T-1}})$ which is impossible by the above observations. We have reached the desired contradiction and thus we have proven the claim.

Let $s$ be as in Claim 1'. Then the Zariski closure of the projection $X_s\to V$ equals $\bar Z$. By Fact \ref{Bezout}, $\deg \bar Z\leq \deg X_s$. The result follows by using (\ref{onx2}), the fact that $s\leq (d-d_0)\alpha_{T-1}-1$, and the same identity used at the end of the proof of Theorem~\ref{maint}.
\end{proof}

We now extend our bound to systems of higher order differential polynomial equations.

\begin{corr}\label{finalcor}
Let $\ell\in \m N$. Suppose $V$ and $S$ are closed subvarieties of $K^n$ and $W$ a subvariety of $K^{n\cdot\alpha_\ell}$. Let $d=\dim V$. If $Z=\{v\in V\setminus S: \nabla_\ell(v)\in W\}$ is finite, then 
$$|Z|\leq (\deg V)^{\alpha_{\ell-1}\alpha_{T'} (m+1)^{d'\alpha_{T'-1}-1}}\;(\deg W)^{\alpha_{T'-1}\cdot\frac{(m+1)^{d'\alpha_{T'-1}}-1}{m}},$$
where $d'=d\cdot\alpha_{\ell-1}$ and $T'=T_1^{m,n\cdot\alpha_{\ell-1}}$ (the latter is defined in (\ref{defT})).
\end{corr}
\begin{proof}
Let $V':=B_{\ell-1}(V)\subseteq K^{n\cdot\alpha_{\ell-1}}$ and $S':=B_{\ell-1}(S)$. For each $i=1,\dots,m$, let $A_i=\{\xi\in \m N^m: |\xi|=\ell-1\text { and } \xi_j=0\text{ for } j<i\}$. If we define
\begin{align*}
W'=\{(x^{\xi},y^{\xi,i})_{\xi\in\Gamma(\ell-1),1\leq i\leq m}\in K^{n(m+1)\alpha_{\ell-1}}\colon 
& y^{\xi,i}=y^{\zeta,j} \text{ when } \xi+{\bf i}=\zeta+{\bf j},\\
& y^{\xi,i}=x^{\xi+{\bf i}} \text{ for } \xi\in\Gamma(\ell-2), 1\leq i\leq m, \\
&(x^\xi,y^{\zeta_i,i})_{\xi\in\Gamma(\ell-1),1\leq i\leq m,\zeta_i\in A_i}\in W \},\\
\end{align*}
then $Z'=\{v\in V'\setminus S': \nabla(v)\in W'\}$ is in bijection with $Z$ and, by Bezout's inequality, $\deg(W')\leq \deg(W)$. By Theorem \ref{maint} (and Remark \ref{imprem} (3)), we have
$$|Z|=|Z'|\leq (\deg V')^{\alpha_{T'} (m+1)^{d'\alpha_{T'-1}-1}}\;(\deg W')^{\alpha_{T'-1}\cdot\frac{(m+1)^{d'\alpha_{T'-1}}-1}{m}},$$
where $d'=\dim V'=d\cdot\alpha_{\ell-1}$ and $T'=T_1^{m,n\cdot\alpha_{\ell-1}}$. Using that $\deg V'\leq (\deg V)^{\alpha_{\ell-1}}$ and $\deg W'\leq \deg W$, the result follows.
\end{proof}

\begin{rem} \label{secondrem} 
\begin{enumerate}
\item In the case when $m=1$, and so $T'=1$, we get $\alpha_0=1$, $\alpha_1=2$, and $\alpha_{\ell-1}=\ell$. Thus, in this case, the bound reduces to 
$$|Z|\leq (\deg V)^{\ell2^{d\ell}}\; (\deg W)^{2^{d\ell}-1},$$
which is precisely the upper bound found in the ordinary case \cite{HPnfcp}.
\item (Without the finiteness assumption on $Z$) When the components of the Zariski closure $\bar Z$ of $Z$ have dimension greater than or equal to $d_0$, we obtain a better bound for $\deg (\bar Z)$ (analogous to Corollary \ref{posdim}) by simply replacing $d'=d\cdot\alpha_{\ell-1}$ by $d'=(d-d_0)\cdot \alpha_{\ell-1}$. In the ordinary case, this bound reduces to 
$$\deg (\bar Z)\leq (\deg V)^{\ell 2^{(d-d_0)\ell}}\;(\deg W)^{2^{(d-d_0)\ell}-1}.$$
\end{enumerate}
\end{rem}

\end{document}